\documentclass[10pt,a4paper]{article}

\usepackage{color}

\newcommand{\EMP}[1]{\textcolor{red}{\emph{#1}}}

\usepackage{url}
\usepackage{verbatim}
\usepackage{latexsym}
\usepackage{amssymb,amstext,amsmath,amsthm}

\DeclareMathOperator\Hom{Hom}
\DeclareMathOperator\End{End}
\DeclareMathOperator\GL{GL}
\newcommand\op{^\mathrm{op}}
\usepackage{epsf}
\usepackage{epsfig}
\usepackage{a4wide}
\usepackage{verbatim}
\usepackage{proof}
\usepackage{latexsym}
\newtheorem{theorem}{Theorem}[section]
\newtheorem{corollary}{Corollary}[section]
\newtheorem{lemma}{Lemma}[section]

\newtheorem{definition}{Definition}[section]

\usepackage{float}
\floatstyle{boxed}
\restylefloat{figure}

\def\oge{\leavevmode\raise
  .3ex\hbox{$\scriptscriptstyle\langle\!\langle\,$}}
\def\feg{\leavevmode\raise
  .3ex\hbox{$\scriptscriptstyle\,\rangle\!\rangle$}}

\newcommand{\inv}{\mathsf{inv}}
\newcommand{\mm}{\mathfrak{m}}

\usepackage{authblk}
\usepackage{doi}

\begin{document}

\title{Azumaya algebras and Barr's Theorem}

\author{Thierry Coquand}
\affil{Computer Science and Engineering Department,
  University of Gothenburg, Sweden,
  \href{mailto:coquand@chalmers.se}{coquand@chalmers.se}
}
\author{Henri Lombardi}
\author{Stefan Neuwirth}
\affil{Université Marie et Louis Pasteur, CNRS, LmB, FR-25000 Besançon, France,
  \href{mailto:henri.lombardi@umlp.fr}{henri.lombardi@umlp.fr},
  \href{mailto:stefan.neuwirth@umlp.fr}{stefan.neuwirth@umlp.fr}
}
\date{\today}
\maketitle


\section*{Introduction}

The goal of this paper is to show how one can use the notion of classifying
topos of a coherent
theory \cite{topos,CLR,Wraith}, here the theory of separably closed local rings
\cite{Wraith},
to prove some basic results about Azumaya algebras in a constructive setting.

A fundamental invariant associated to a commutative field $F$ is its \EMP{Brauer group}, 
whose elements are the division algebras over $F$, up to isomorphism.
Alternatively an element of the Brauer group can be
described as a central simple algebra, up to Morita equivalence, and the group operation is then the tensor
product. Two important characterisations of the notion of central simple algebra are the
following:\footnote{We study these characterisations in a constructive setting in the reference \cite{CLN}.}
\begin{enumerate}
\item an algebra $A$ over $F$ which is finitely generated
  and such that the canonical map $A\otimes_F A\op\rightarrow \End_F(A)$
  is an isomorphism; or, equivalently,
\item an algebra $A$ which becomes a matrix algebra over an algebraic extension of $F$, or, equivalently, which
  becomes a matrix algebra over an algebraic and separable extension of $F$.
\end{enumerate}

This notion of central simple algebra over a \EMP{field} has been generalised to the notion of
Azumaya algebra over an arbitrary \EMP{commutative ring} $k$. The definition is a generalisation of the first
characterisation: an algebra $A$ which is finitely generated projective over the ring~$k$ 
and such that the canonical map $A\otimes_k A\op\rightarrow \End_k(A)$
is an isomorphism. In this paper, we show in a constructive setting that this definition
is equivalent to a suitable generalisation of the second characterisation, using a constructively valid
version of Barr's Theorem \cite{topos}. 

\section{Some coherent theories}

We consider some coherent theories \cite{topos} in the language of rings.

\subsection{Theory of local rings}

The theory of \EMP{local rings} can be formulated in a coherent way with the axioms
$$0\neq 1\qquad\qquad\inv(x)\vee\inv(1-x)$$
where $\inv(x)$ denotes $\exists_y~1 = xy$.

If we add to our language
a new predicate $\mm(x)$ with axioms
$$\inv(x)\vee \mm(x)\qquad\qquad\neg (\mm(x)\wedge \inv(x))$$
then we get the theory of local rings that are \EMP{residually discrete},\footnote{The terminology ``discrete''
  for a set with a decidable equality, while conflicting with some topological intuitions, is well established in
  constructive mathematics since the work of Bishop \cite{Bishop}, and we follow this usage in this paper.} i.e.\ where
the property of being invertible is now decidable, and the predicate $\mm(x)$
corresponds to the unique maximal ideal of the ring. In this theory, we can express
in a coherent way
that the ring is Henselian by stating that for any monic polynomial $f$,
if we have $u$ such that $f(u)\in \mm$ and $f'(u)\notin \mm$ then we can find $\alpha$
such that $f(\alpha) = 0$ and $\alpha-u \in \mm$. Equivalently \cite{ALP},
any monic polynomial of degree $l+1$
which is of the form $X^l(X-1)$ mod.~$\mm$
has a root which is equal to $1$ mod.~$\mm$.
The work \cite{ALP} shows in a constructive setting for this theory
the lifting property for idempotents of a finitely generated commutative algebra over
a Henselian local ring.

It does not seem possible however to express the
property of being Henselian in a coherent way in the language of local rings (i.e.\ without introducing
the predicate $\mm(x)$).
Surprisingly, as observed by G. Wraith \cite{Wraith}, the property of being
\EMP{separably closed}, corresponding classically to the property of being Henselian with the residue field being separably closed, \EMP{can be} expressed in a coherent way.
In this paper, we present some remarks connected to this observation.
The first one is a small variation of Wraith's axiomatisation.
Both axiomatisations can be used to define the étale site over a ring.
Finally, we show in a constructive setting
that Azumaya algebras are exactly the algebras
that are locally matrix algebras for the étale topology.

\subsection{Theory of separably closed  local rings}

If $R$ is a ring, and $a_1,\dots,a_n$ elements of $R$, we write $(a_1,\dots,a_n)$ for the ideal of $R$ generated
by $a_1,\dots,a_n$.

We define, following Joyal \cite{Joyal}, the \emph{Zariski spectrum} of $R$ as the universal support
$D\colon R\rightarrow Z(R)$ over $R$. A \emph{support} is a map $d\colon R\rightarrow M$ into a bounded distributive lattice $M$
satisfying the relations
$$
d(1) = 1\qquad\qquad d(0) = 0\qquad\qquad d(a+b)\leqslant d(a)\vee d(b)\qquad\qquad d(ab) = d(a)\wedge d(b)
$$
We write $d(a_1,\dots,a_n)$ for $d(a_1)\vee\cdots\vee d(a_n)$. It is direct that $d(a)\leqslant d(a_1,\dots,a_n)$
as soon as $a$ belongs to the radical of the ideal $(a_1,\dots,a_n)$. Conversely, it is elementary to check
that we get a support if we define $D(a_1,\dots,a_n)$ to be the radical of the ideal $(a_1,\dots,a_n)$, and this
is a concrete way to build the universal support.

In particular, this shows that $D(a_1,\dots,a_n) = 1$ if, and only if, $(a_1,\dots,a_n) = 1$.

It is also direct, for any support $d$, to see that $d(a,b) = d(a+b,ab)$ and more generally
that
$$d(a_1,\dots,a_n) = d(c_1,\dots,c_n)\text,$$
where  $c_1,\dots,c_n$ are the elementary symmetric functions of $a_1,\dots,a_n$.

\medskip

For the formulation of the theory of separably closed local rings, we
need the notion of universal
splitting algebra \cite[Sect.~III-4]{LQ} $L = R[x_1,\dots,x_n] = R[X_1,\dots,X_n]/I$ of
a monic polynomial
$$f = X^n-u_1 X^{n-1}+u_2X^{n-2}-\cdots$$
in $R[X]$, where $I$ is the ideal generated by 
$u_1 - \sigma_1$, \dots, $u_n - \sigma_n$, with $\sigma_1$, \dots, $\sigma_n$ in $R[X_1,\dots,X_n]$ the elementary symmetric
polynomials.
We know, using Cauchy modules \cite[Fact~III-4.3]{LQ}, that $L$ is freely generated of rank $n!$ as a module over $R$
and it follows from this that $L$ is a faithfully flat extension of~$R$.

For the notion of faithfully flat extension, we follow the definition in \cite[Sect.~VIII-6]{LQ}: it is a commutative $R$-algebra $i\colon R\rightarrow S$
such that

\begin{itemize}
\item if $u$ is a row vector in $R^n$ and $v$ a column vector in $S^n$ such that $i(u)v = 0$ then we can find a matrix $M$ with
  coefficients in $R$ and $w$ such that $uM = 0$ and $v = i(M)w$;
\item if $i(a)$ belongs to the ideal $(i(a_1),\dots,i(a_n))$ in $S$ then $a$ belongs to the ideal $(a_1,\dots,a_n)$ in $R$.
\end{itemize}

It then follows \cite[Sect.~VIII-6]{LQ} that $i$ is injective. It is also clear from
this definition that $S$ is faithfully flat over $R$
when it is a finitely generated free $R$-module of rank $>0$.

\begin{sloppypar}
  We define $\delta_i(f)$ to be the element in $R$ equal to
  $\sigma_i(f'(x_1),\dots,f'(x_n))$ and $\Delta(f)$ to be $D(\delta_1(f),\dots,\delta_n(f))$.
\end{sloppypar}

We say that $f$ is \emph{unramifiable} if, and only if, $1 = \Delta(f)$.
Since
$\Delta(f)=D(f'(x_1),\dots,f'(x_n))$, 
this condition is equivalent to $(f'(x_1),\dots,f'(x_n)) = 1$ in the universal
splitting algebra $L$.

Over a residually discrete local ring with maximal ideal $\mm$,
to be unramifiable means that the polynomial $f$ has a simple root in some extension of the residue field $R/\mm$.
This holds in classical mathematics, but also constructively
with a dynamical interpretation of extensions of $R/\mm$.

\begin{lemma}\label{simple}
  Let $R$ be a local ring, residually discrete with maximal ideal $\mm$ and residue field $R/\mm$.
  If $g$ is a monic polynomial in $R[X]$ which has a simple root in some nontrivial commutative $R/\mm$-algebra, then
  $g$ is unramifiable.
\end{lemma}

\begin{lemma}[{\cite[Lemme IX-7.1]{LQ2}}]\label{unram}
  If $g = (X-a_1)\cdots (X-a_p) f$ in $R[X]$ with $f$ monic of degree $n$, $\Delta(g) = 1$, and $R$
  is local, then $g'(a_1)$ is invertible or \dots\ or $g'(a_p)$ is invertible or
  $\Delta(f) = 1$.
\end{lemma}  

\begin{proof}
  We write $h = (X-a_1)\cdots (X-a_p)$.
  
  We have in $R[x_1,\dots,x_n]$, the universal splitting algebra of $f$,
  $$1 = D(g'(a_1),\dots,g'(a_p), h(x_1)f'(x_1),\dots,h(x_n)f'(x_n))$$
  and hence
  $$1 = D(g'(a_1),\dots,g'(a_p), f'(x_1),\dots,f'(x_n))$$
  and from this follows
  $$1 = D(g'(a_1),\dots,g'(a_p))\vee \Delta(f)$$
  in $R[x_1,\dots,x_n]$. 
  Since the universal splitting algebra is faithfully flat, we get
  $$1 = D(g'(a_1),\dots,g'(a_p))\vee \Delta(f)$$
  in $R$, and since $R$ is local we have the conclusion.
\end{proof}  

\begin{definition}[{\cite[Définition IX-7.2]{LQ2}}]\upshape
  A local ring~$R$ is \emph{separably closed} if every unramifiable monic polynomial has a root in~$R$.
\end{definition}

\begin{lemma}[{\cite[Lemme IX-7.3]{LQ2}}]\label{other}
  If $R$ is a separably closed local ring then
  any unramifiable monic polynomial has a simple root in $R$.
\end{lemma}

\begin{proof}
  Assume that $R$ is separably closed local.
  Let $g$ be a monic unramifiable polynomial. 
  It has a root $a_1$ and we can write
  $g = (X-a_1)f_1$. By Lemma~\ref{unram}, $a_1$ is a simple root of $g$ or $\Delta(f_1) = 1$. 
  If $\Delta(f_1) = 1$ then it has a root $a_2$ and $g = (X-a_1)(X-a_2)f_2$. Then by Lemma~\ref{unram}
  again, $a_1$ or $a_2$ is a simple root of $g$ or $\Delta(f_2) = 1$, and so on until we find
  a simple root of~$g$.
\end{proof}  




\medskip

Let $k$ be a ring. We can consider the following two coherent theories for a commutative $k$-algebra $R$:

\begin{enumerate}
\item $T_1$ states that $R$ is local, and that any monic unramifiable polynomial has a root.

\item $T_2$ states that $R$ is local, and that any monic unramifiable polynomial has a \emph{simple} root.
\end{enumerate}

We have just seen that these two theories are equivalent.

The equivalence between these two theories can be
seen as a logical formulation of the ``trick'' attributed to O. Gabber described
in \cite[\href{https://stacks.math.columbia.edu/tag/02LI}{Tag 02LI}]{Stacks}.

A morphism
\[
S \longrightarrow S[X]/(f),
\]
with \(f\) monic, is finite locally free, and thus represents the
``finite locally free'' stage in Gabber's refinement of an étale covering.
Étaleness is obtained by further localising, that is,
by passing to the Zariski covering
\[
S \longrightarrow S[X]/(f)[1/f'].
\]

Lemma~\ref{other}  provides the logical ingredient of this refinement: over a local
base ring, the existence of a root of a monic unramifiable polynomial implies the
existence of a simple root. Internally, this shows that coverings generated by
finite locally free morphisms of the above form can be refined, after a Zariski
localisation, to étale morphisms. In this sense, the equivalence between the
two theories $T_1$ and $T_2$ is the internal counterpart of Gabber's trick:
finite locally free coverings suffice to generate the same local objects as
étale coverings, up to Zariski refinement.

\medskip

These two coherent theories define two different \emph{sites} (notions of covering), in the sense of Grothendieck (see e.g.\ \cite{topos}), over the opposite of
the category of finitely presented
commutative $k$-algebras \cite{topos}. The first theory $T_1$ corresponds to:

\begin{enumerate}
\item $S$ is covered by $S\rightarrow S[1/s_1]$, \dots, $S\rightarrow S[1/s_n]$ if $1 = (s_1,\dots,s_n)$;
\item $S$ is covered by $S\rightarrow S[X]/(f)$ if $f$ is monic and unramifiable.
\end{enumerate}

The second theory $T_2$ corresponds to:

\begin{enumerate}
\item $S$ is covered by $S\rightarrow S[1/s_1]$, \dots, $S\rightarrow S[1/s_n]$ if $1 = (s_1,\dots,s_n)$;
\item $S$ is covered by $S\rightarrow S[X]/(f)[1/f']$ if $f$ is monic and unramifiable.
\end{enumerate}

Both define the same corresponding \emph{classifying topos} of the
theory of separably closed local commutative $k$-algebras \cite{topos}.
In this sheaf model, we have a ``generic'' local commutative $k$-algebra $R$ which is separably closed. 

We also may state, following Wraith,

\begin{theorem}[{\cite[Lemmes IX-7.4, IX-7.5]{LQ2}}]\label{separable}
  If $R$ is a local ring which is also residually discrete, then $R$ is separably closed if, and only if,
  $R$ is Henselian and its residue field is separably closed.
\end{theorem}

\begin{proof}
  Let $g$ be a monic polynomial which has residually a simple root $u$. We show by induction on the degree of $g$
  that $g$ has a root in $R$ which is equal to $u$ mod.~$\mm$.
  The polynomial $g$ is unramifiable by Lemma~\ref{simple}, since it has residually a simple root,
  and it has a root $a$. If $a  \neq u$ mod.~$\mm$ then we write
  $g = (X-a)f$, and $u$ is residually a simple root of $f$.
\end{proof}

Thus a residually discrete local ring is separably closed if and only if it is ``strictly Henselian''.

\section{Some lemmas about Henselian rings}

In the next section, we need a generalisation of a result about Henselian rings. In order to make the paper self-contained
we mention here this result, Theorem~\ref{thm1}, taken from the reference \cite{ALP}.

We assume in this section that $R$ is local residually discrete with maximal ideal $\mm$ and Henselian.

Lemma~\ref{lem1} is a main step in the proof of Theorem~\ref{thm1}: we provide a variation of the proof that requires less constructive Galois theory than the original one in \cite{ALP}.

\begin{lemma}[{\cite[Prop.~5.7]{ALP}}]\label{lem1}
  Let $f$ be a monic polynomial  of degree $n$
  in $R[X]$ and $R[x] = R[X]/(f)$. If we have $e(X)$ in $R[X]$ such that $e(x)$
  is idempotent
  mod.~$\mm R[x]$, then we can find $u$ idempotent in $R[x]$ with $u = e(x)$ mod.~$\mm R[x]$.
\end{lemma}

\begin{proof}    
  Let $L = R[x_1,\dots,x_n]$ be the universal splitting algebra of $f$.
  This is
  a free $R$-module over $R$ of dimension $n!$ and the symmetric group acts on it \cite[Fact~III-4.4]{LQ}.
  We may identify $R[x]$ with the subalgebra $R[x_1]$ and we have
  $e(x_1)$ in $R[x_1]$ which is idempotent mod.~$\mm R[x_1]$. We show that we have
  $u_1$ in $R[x_1]$ which is idempotent in $R[x_1]$ and equal to $e(x_1)$ mod.~$\mm R[x_1]$.

  The product $e(x_1)\cdots e(x_n)$ is in $R$ and it is an idempotent mod.~$\mm$, so it is $1$ or $0$ mod.~$\mm$.
  If it is $1$, we have $e(x_1) = 1$ mod.~$\mm R[x_1]$ and we can take $u_1 = 1$.

  We can thus assume $e(x_1)\cdots e(x_n) = 0$ mod.~$\mm$.
  We write $e_I = \prod_{i\in I}e(x_i)$ for $I$ subset of $1,\dots,n$.
  We let\footnote{The $e_I,~|I|=k$, are pairwise conjugate by the action of the symmetric group
    and equality is decidable mod.~$\mm L$.}
  $k<n$ be such that $e_I \neq 0$ mod.~$\mm L$ if $|I| = k$ and $e_I = 0$ mod.~$\mm L$ if $|I| = k+1$.
  Note that the $e_I$
  for $|I| = k$ are pairwise orthogonal mod.~$\mm L$.
  We have $\Sigma_{|I|=k} e_I = 1$ mod.~$\mm$ since it is an element in $R$ which is idempotent mod.~$\mm$.
  Hence the family of the $e_I,~|I| = k$, forms an
  FSOI\footnote{A ``Fundamental System of Orthogonal Idempotents'' (\cite[Sect.~II-4]{LQ}).}
  mod.~$\mm L$.

  We deduce that
  $$g(X) = \prod_{|I|=k} (X-e_I)\text,$$
  which is a monic polynomial in $R[X]$, is of the form $X^{N-1}(X-1)$ mod.~$\mm R[X]$ with $N = \binom{n}{k}$.

  Since $R$ is Henselian, we have $\alpha$ in $R$, $\alpha = 1$ mod.~$\mm$, such that $g(\alpha) = 0$; furthermore $\lambda = g'(\alpha)$ is
  invertible with $\lambda = 1$ mod.~$\mm$. If we write $g_I = g/(X-e_I)$ then $g'(X) = \Sigma_I g_I$. We take
  $$v_I = \frac{1}{\lambda}g_I(\alpha)$$
  and we then have that the $v_I,~|I| = k$, form an
  FSOI in $R[x_1,\dots,x_n]$. We have $v_I = \Pi_{J\neq I} (1-e_J) = e_I$ mod.~$\mm L$.
  Let $\mathcal I_1$ be the set of all $I$ such that $1$ is in $I$ and $|I| = k$.
  If we write $u_1$ for $\Sigma_{I\in \mathcal I_1} v_I$,
  we have $u_1$ in $R[x_1]$ and $u_1$ is equal to $\Sigma_{I\in \mathcal I_1} e_I = e(x_1)$ mod.~$\mm L$.
  This equality follows from $e(x_1) = e(x_1)(\Sigma_I e_I) = \Sigma_I e(x_1)e_I = \Sigma_{I\in \mathcal I_1}e_I$ mod.~$\mm L$.
  Hence $e(x_1) - u_1\in R[x_1]\cap \mm L = \mm R[x_1]$.
\end{proof}


\begin{theorem}[{\cite[Theorem~5.9]{ALP}}]\label{thm1}
  If $A$ is a commutative
  $R$-algebra which is finitely generated as an $R$-module and $a$ in $A$ is idempotent mod.~$\mm A$,
  then we can find $e$ in $A$ idempotent in $A$ and such that $e = a$ mod.~$\mm A$.
\end{theorem}

The following generalisation is not in \cite{ALP} and will be crucial in the proof of Lemma~\ref{main1}.

\begin{corollary}\label{corthm1}
  If $A$ is a (not necessarily commutative)
  $R$-algebra which is finitely generated as an $R$-module and $a$ in $A$ is idempotent mod.~$\mm A$,
  then we can find $e$ in $A$ idempotent in $A$ and such that $e = a$ mod.~$\mm A$.
\end{corollary}

\begin{proof}
  Write $\varphi(x) = x^2 -x$ and consider the finitely generated algebra
  $B=R[a]$. Since $\varphi(a) = a^2-a$ is in $\mm A$, the constant term of its characteristic polynomial is in $\mm A\cap R=\mm$, so that we have $N$ such that $\varphi(a)^N$ is in $\mm B$. 
  We can then consider the following Newton
  quadratic iteration of $a$ inside $B$ by defining
  inductively $u_0 = a$ and $u_{n+1} = u_n - (2u_n-1)\varphi(u_n)$ (compare \cite[Corollary~III.10.4]{LQ}). One checks that $u_n=a$ mod.~$\mm A$ and that $\varphi(u_n)$ is in $\varphi(u_{n-1})^2B$, so that it is in $\mm B$ for $n$ big enough, i.e.\ $u_n$ is idempotent
  mod.\ $\mm B$. We can now use Theorem~\ref{thm1} for the commutative $R$-algebra $B$
  to lift $u_n$ to an idempotent $e$ in $B$ such that
  $e = u_n$ mod.\ $\mm B$. We also have $e = a$ mod.\ $\mm A$.
\end{proof}

\section{Azumaya algebras}

An \emph{Azumaya algebra} $A$ over a ring $k$ is an algebra such that
\begin{enumerate}
\item $A$ is finitely generated projective;
\item the canonical map $A\otimes_k A\op\rightarrow \End_k(A)$ is an isomorphism.
\end{enumerate}

Note that we allow the trivial algebra to be an Azumaya algebra. Note also that any matrix algebra is Azumaya.

In this section $k$ will denote a commutative base ring, and $R$ will denote the ``generic'' separably closed
commutative $k$-algebra, model of the coherent theory $T_1$ described in the first section.

One can equivalently describe the algebra structure by a bilinear map $\mu\colon A\times A\rightarrow A$;
its scalar extension along $k\rightarrow S$ is the unique $S$-algebra $A_S$ through which every bilinear map extending 
$\mu$ factors. Thus all subsequent arguments can be phrased without explicit tensor notation.

Over a finitely generated free $k$-module, an algebra structure is given concretely by a multiplication table.
A module $A$, finitely generated projective
over~$k$, is given concretely by a projection matrix over $k$ \cite[Chap.~I]{LQ}, and
we can similarly represent the algebra structure by a multiplication table (so $m^3$~elements of~$k$ if the projection
matrix has finite size~$m$). We can then compute \cite[Theorem~X-1.6]{LQ} comaximal elements $s_1,\dots,s_l$ such that $A[1/s_i]$
is free over $k[1/s_i]$, and the second condition for being an Azumaya algebra becomes then that \(l\)~determinants
are invertible.\footnote{After localisation to $s_i$ the Azumaya algebra becomes free of rank $r_i=n_i^2\leqslant m$.}
From this discussion follows that if $A$ is an Azumaya algebra, given by a multiplication table,
and we have a commutative $k$-algebra $k\rightarrow S$, then the algebra $A_S$ over $S$, with the same multiplication table, is also
an Azumaya algebra.

We have the following characterisation, which generalises the result that
an algebra is central simple over a field if, and only if, it can be split by a separable extension of this field.

\begin{theorem}\label{main}
  $A$ is an Azumaya algebra over $k$ if, and only if, we can build a finite covering tree rooted at~$k$
  using the following, where at every leave the scalar extension $A_S$ is a matrix algebra:
  \begin{enumerate}
  \item $S$ is covered by $S\rightarrow S[1/s_1]$, \dots, $S\rightarrow S[1/s_n]$ if $1 = (s_1,\dots, s_n)$;
  \item $S$ is covered by $S\rightarrow S[X]/(f)$ if $f$ is monic and unramifiable.
  \end{enumerate}
\end{theorem}

We shall use the local-global principle and the descent principle with respect to faithfully flat extension
for projective modules (\cite[Theorem~VIII-6.7]{LQ}). For completeness, we recall the statement. Note that the proof
in \cite{LQ} is constructive, using the concrete definition of faithfully flat extension and the fact that, also
constructively, a module is finitely generated projective if, and only if, it is finitely presented and flat.

\begin{theorem}[Descent]\label{descent}
  Let $A\rightarrow B$ be a faithfully flat extension and $M$ an $A$-module. If $M\otimes_A B$ is finitely generated
  projective, then
  so is $M$. If we have $1 = (s_1,\dots,s_n)$ in $A$ and each $M[1/s_i]$ is finitely presented and projective
  over $A[1/s_i]$ then so is $M$.
\end{theorem}

We prove first that if we have such a covering tree then the algebra $A$ is Azumaya.

If we have such a covering tree rooted at~$k$, notice that each extension
$S\rightarrow S[X]/(f)$ for $f$ monic is faithfully flat, and
$S\rightarrow \prod_i S[1/s_i]$ is also faithfully flat if $1 = (s_1,\dots, s_n)$.
It follows then that $A$ is finitely generated projective by Theorem~\ref{descent}.
The map $A\otimes_k A\op\rightarrow \End_k(A)$ is then a bijection since it becomes
a bijection by a faithfully flat extension. (As we saw before, to be a bijection is witnessed
by the fact that some determinants are invertible, and an element is invertible if it becomes
invertible by a faithfully flat extension.)

For the converse, we look at the algebra $A$ (given concretely by its multiplication table) inside
the étale topos over $k$, using the site defined in the first section. (This is similar to the constructive
use of the algebraic closure in \cite{Mannaa}.) The generic ring $R$, defined by $R(S) = S$ for $S$ a finitely
presented commutative $k$-algebra, is, in the internal logic of this topos, a separably closed local ring. The algebra $A$
defines then an $R$-algebra $A_R(S) = A_S$, which is the Azumaya algebra given by the same
multiplication table as the one for $A$. If we can prove in intuitionistic logic that
an Azumaya algebra over a separably closed ring is a matrix algebra, we will get, using the interpretation of
intuitionistic logic in a sheaf model \cite{topos,Mannaa}, a covering tree rooted at~$k$ which yields Theorem~\ref{main}.

Internally in the classifying topos of separably closed local commutative $k$-algebras, the object $R$ is separably closed local
and the claim that $A_R$ is a matrix algebra is a coherent statement, provable in intuitionistic logic.
Such a proof yields externally a finite covering tree rooted at~$k$ satisfying the conditions above.

The converse is a consequence of the following two results.

\begin{lemma}\label{main1}
  If $A$ is an Azumaya algebra over $R$, which is a residually
  discrete  separably closed local commutative $k$-algebra, then $A$ is a matrix algebra over $R$.
\end{lemma}

\begin{proof}
  We know that $R$ is Henselian and that $R/\mm$ is separably closed by Theorem~\ref{separable}.

  We simplify (and correct slightly) the argument in Milne's course notes
  (\cite[Chapter IV, Proposition 1.6]{Milne}),
  using the constructive
  development in \cite{ALP} and \cite{CLN}. If $\mm$ is the maximal ideal of $R$, we have that $A$
  is split over $R/\mm$ (\cite{CLN}), and hence we have a matrix algebra decomposition $e_{ij}$
  mod.~$\mm A$. We lift the idempotent $e_{11}$
  mod.~$\mm A$ to some idempotent $e$ in $A$ using Corollary~\ref{corthm1}.

  The $R$-module $Ae$ is finitely generated projective and hence free since $R$ is a local ring.
  The map $A\rightarrow \End_R(Ae)$ is then an isomorphism, since it is a map between
  two free $R$-modules   of same dimension, and it is residually an isomorphism.
\end{proof}

The next proof uses Barr's Theorem in a way similar to \cite{Penon,Wraith}, but the difference
is that we only need a constructively valid version of this theorem.

\begin{corollary}
  If $A$ is an Azumaya algebra over $R$, which is a separably closed local commutative $k$-algebra,
  then $A$ is a matrix algebra over $R$.
\end{corollary}

\begin{proof}
  For a given Azumaya algebra, the conclusion is coherent \cite{topos,Mannaa}, and the hypothesis that $A$
  is an Azumaya algebra is formulated in a coherent theory. Since it is proved with the extra
  hypothesis that $R$ is residually discrete, the syntactic version of Barr's Theorem\footnote{We can actually
    follow the argument in \cite{topos} and apply Friedman-Dragalin's translation
    \cite{Avigad} with respect to the proposition
    expressing that $A$ is a matrix algebra.}
  (constructively valid) shows that it can be proved without this hypothesis.
\end{proof}

Theorem~\ref{main} follows by interpreting this result in the classifying topos of the theory of
separably closed local commutative $k$-algebras.

The statement of Theorem~\ref{main} can be seen as a generalisation of the fact that over a field, a
central simple algebra becomes a matrix algebra by a faithfully flat extension \cite{CLN} (classically
over an algebraic extension).

\medskip

We can use such a covering tree to give arguments by ``tree induction''. We give some simple examples.

\begin{lemma}\label{central}
  Any Azumaya algebra $A$ is central.
\end{lemma}

\begin{proof}
  This means that if $a$ in $A$ satisfies $ax = xa$ for all $x$ in $A$ then $a$ is in $k$. We prove
  this by tree induction using Theorem~\ref{main}. The statement holds on the leaves (it holds for
  a matrix algebra) and descends to the root.
\end{proof}

\begin{lemma}\label{automorphism}
  If $R$ is a local ring and $A$ a matrix algebra of rank $\geqslant 1$
  \cite[Definition~X-2.2]{LQ}
  then any automorphism of $A$ is of the form
  $x\mapsto axa^{-1}$ for some $a$ in $\GL_n(A)$.
\end{lemma}

\begin{proof}
  Let $A = \End_R(R^n)$ with $v_1,\dots,v_n$ the canonical basis of $R^n$ and define $e_{ij}v_k = \delta_{jk}v_i$.
  Let $\psi$ be an automorphism of $A$. If $p_{ij} = \psi(e_{ij})$, we have
  that the $p_i = p_{ii}$ form an FSOI and $p_{ij}p_{kl} = \delta_{jk}p_{il}$.
  Since $R$ is local, each $V_i = p_i(R^n)$ is free.  Furthermore $R^n$ is a direct sum of the $V_i$'s; the $V_i$'s
  are pairwise isomorphic since $p_{ji}$ defines an isomorphism between $V_i$ and $V_j$.  So they are all free of
  rank $1$.
  Let $w_1$ be a basis of $V_1$ and $w_j = p_{j1}w_1$. If $a(v_i)=w_i$ then we have $\psi(x) = axa^{-1}$.
\end{proof}

\begin{theorem}[Skolem-Noether]
  If $A$ is an Azumaya algebra, of rank $\geqslant 1$
  as a projective module over a ring $k$,
  and $\psi$ is an automorphism of $A$, then
  $M = \{ a\in A~|~\forall_x\, ax = \psi(x)a\}$ is a projective $k$-module of rank $1$.
\end{theorem}

\begin{proof}
  We use Theorem~\ref{main}. The statement holds on the leaves ($M$ is even \EMP{free} of rank $1$
  on the leaves by Lemma~\ref{automorphism}) and descends to the root by
  tree induction using Theorem~\ref{descent}.
\end{proof}

If $k$ is local, the module $M$ is free of rank $1$ and any generator $u$ of $M$ is invertible (since it is
invertible on the leaves) and $\psi$ is the inner automorphism defined by $u$, so we recover the usual
form of the Skolem-Noether Theorem (any automorphism of a central simple algebra is an inner automorphism).

Claude Quitté has suggested the following refinement.

\begin{theorem}
  If $A$ is an Azumaya algebra, of rank $\geqslant 1$ as a projective module over a ring $k$,
  and $\psi$ is an automorphism of $A$, then
  $$M = \{ a\in A~|~\forall_x\, ax = \psi(x)a\}\quad\text{and}\quad
  N = \{ a\in A~|~\forall_x\, a\psi(x) = xa\}$$
  are      projective $k$-modules of rank $1$ and $M\otimes_k N\simeq k$ and $MN = k1$.
  Furthermore, $M$ and $N$ are direct factors in $A$.
\end{theorem}

\begin{proof}
  We have on the leaves $M_SN_S \subseteq S 1$ and this statement descends to the root by tree induction.
  This can be used to define a map $M\rightarrow \Hom(N,k),~x\mapsto (y\mapsto xy)$. This map
  is an isomorphism on the leaves, and this statement also descends to the root by tree induction.
  Finally, $M$ and $N$ are direct factors on the leaves, and this fact descends as
  well.
\end{proof}

\section{Theory of algebraically closed local rings}

To get the theory of algebraically closed local rings, we add to the theory of local rings the axioms
$$\exists_x~f(x) = 0$$
for any monic nonconstant polynomial $f$.

Wraith conjectured in \cite{Wraith} that this should define the classifying theory of the
fppf topos. This question is discussed further in Blechschmidt's Ph.D. thesis \cite{Ingo}.

The same argument as in Lemma~\ref{main1} shows the following result.

\begin{lemma}
  If $A$ is an Azumaya algebra over $R$, which is a residually
  discrete algebraically closed local commutative $k$-algebra, then $A$ is a matrix algebra.
\end{lemma}


We can then prove as in the previous section
\begin{theorem}\label{main2}
  $A$ is an Azumaya algebra over $k$ if, and only if, we can build a finite covering tree rooted at~$k$
  using the following, where all leaves are matrix algebras:
  \begin{enumerate}
  \item $S$ is covered by $S\rightarrow S[1/s_1]$, \dots, $S\rightarrow S[1/s_n]$ if $1 = (s_1,\dots, s_n)$;
  \item $S$ is covered by $S\rightarrow S[X]/(f)$ if $f$ is monic nonconstant.
  \end{enumerate}
\end{theorem}

\section{Conclusion}

This paper is a beginning of the theory of Azumaya algebras in a constructive setting. The natural next step
will be to analyse Gabber's Theorem  \cite{Kervaire} that the group of Azumaya algebras 
is the torsion part of the second cohomology group with coefficients in $\GL_n(k)$ for the étale topology.

\end{document}